\documentclass[a4paper, 12pt]{amsart}

\usepackage{amsmath}
\usepackage{amssymb}
\usepackage{ascmac}
\usepackage{amsthm}
\usepackage{esint}

\makeatletter

\@addtoreset{equation}{section}
\makeatother 

\theoremstyle{plain}
\newtheorem{thm}{Theorem}[section]
\newtheorem*{thm*}{Theorem}
\newtheorem{prop}[thm]{Proposition}
\newtheorem{lem}[thm]{Lemma}

\theoremstyle{definition}

\theoremstyle{remark}
\newtheorem{rem}[thm]{Remark}

\renewcommand{\epsilon}{\varepsilon}

\newcommand{\lip}{\operatorname{lip}}

\newcommand{\mm}{{\mathfrak m}}
\newcommand{\eps}{\varepsilon}
\newcommand{\restr}[1]{\lower3pt\hbox{$|_{#1}$}}
\newcommand{\kse}{{\sf ks}}
\newcommand{\E}{{\sf E}}
\newcommand{\lims}{\varlimsup}
\newcommand{\limi}{\varliminf}
\newcommand{\ks}{{\sf KS}}
\newcommand{\Grad}{\operatorname{Grad}}
\newcommand{\X}{{\rm X}}
\newcommand{\Y}{{\rm Y}}
\newcommand{\B}{{\rm B}}

\newcommand{\sfd}{{\sf d}}
\newcommand{\dfr}{{\mathfrak d}}
\newcommand{\G}{{\sf G}}

\newcommand{\CAT}{\operatorname{CAT}}

\newcommand{\e}{{\rm e}}
\renewcommand{\d}{{\rm d}}
\renewcommand{\o}{{\rm o}}

\newcommand{\RCD}{\operatorname{RCD}}
\newcommand{\SMCPBG}{\operatorname{SMCPBG}}
\newcommand{\Cub}{\operatorname{Cub}}

\DeclareMathOperator*{\aplims}{{\rm ap\,-}\lims}

\DeclareMathOperator{\aplip}{{\rm ap\,-lip}}
\newcommand{\md}{{\sf md}}
\newcommand{\sn}{{\sf sn}}
\newcommand{\D}{{\sf D}}
\newcommand{\nn}{{\mathfrak n}}

\usepackage{latexsym}

\makeatletter
\@namedef{subjclassname@2020}{%
  \textup{2020} Mathematics Subject Classification}
\makeatother

\title[Dirichlet problem for harmonic maps]{Dirichlet problem for harmonic maps from strongly rectifiable spaces into regular balls in $\CAT(1)$ spaces}

\author{Yohei Sakurai}
\address{Department of Mathematics, Saitama University, 255 Shimo-Okubo, Sakura-ku, Saitama-City, Saitama, 338-8570, Japan}
\email{ysakurai@rimath.saitama-u.ac.jp}

\date{August 12, 2023}

\begin{document}

\begin{abstract}
In this note,
we study the Dirichlet problem for harmonic maps from strongly rectifiable spaces into regular balls in $\CAT(1)$ space.
Under the setting,
we prove that
the Korevaar-Schoen energy admits a unique minimizer. 
\end{abstract}

\maketitle

\section{Introduction}

\subsection{$\CAT(0)$ target}\label{eq:cat0}

The study of harmonic maps between singular spaces is one of the central topics in geometric analysis since the pioneering work by Gromov-Schoen \cite{GS}.
In an earlier stage,
the theory has been developed by Korevaar-Schoen \cite{KS}, Jost \cite{J1}, \cite{J2} and Lin \cite{L},
independently.
Korevaar-Schoen \cite{KS} have introduced an energy for $L^2$-maps from a Riemannian domain into a complete metric space,
and derived its basic properties;
for instance,
the existence of energy density,
the consistency of Sobolev functions,
and its lower semicontinuity.
Furthermore,
they have proven that
if the target space is a $\CAT(0)$ space (i.e., non-positively curved space in the sense of A. D. Alexandrov),
then the energy admits a unique minimizer for the Dirichlet problem.

One of the research directions is to generalize the theory for non-smooth source spaces.
Such attempts have been done by Gregori \cite{G}, Eells-Fuglede \cite{EF}, Kuwae-Shioya \cite{KuS}, and so on.
Gregori \cite{G} and Eells-Fuglede \cite{EF} have dealt with a domain of a Lipschitz manifold and that of a Riemannian polyhedron,
respectively.
Kuwae-Shioya \cite{KuS} have examined a metric measure space satisfying the so-called strongly measure contraction property of Bishop-Gromov type, called $\SMCPBG$ space.
A typical example of $\SMCPBG$ space is an Alexandrov space,
which is a metric space equipped with the notion of a lower sectional curvature bound.

In recent years,
the theory of metric measure spaces with a lower Ricci curvature bound has been vastly developed.
In the literature,
the main research object is the so-called $\RCD(K,N)$ space,
which has been introduced by Ambrosio-Gigli-Savar\'{e} \cite{AGS} and Gigli \cite{G00}.
It is natural to ask whether the Korevaar-Schoen theory can be extended to $\RCD(K,N)$ spaces.
Here we notice that
an $\RCD(K,N)$ space is not necessarily a $\SMCPBG$ space (especially, in the collapsed case) since the Bishop type inequality is required in the $\SMCPBG$ condition,
and hence the framework of Kuwae-Shioya \cite{KuS} does not cover $\RCD(K,N)$ spaces.
Gigli-Tyulenev \cite{GT2} were able to work in a more general context,
covering the $\RCD$ setting.
In their framework,
source spaces are assumed to be locally uniformly doubling, strongly rectifiable and to satisfy a Poincar\'e inequality (more precisely, see Subsection \ref{sec:PI}).
Under this setting,
they have defined a Korevaar-Schoen type energy and deduced some fundamental results such as the existence of energy density and its lower semicontinuity.
They have concluded the solvability of the Dirichlet problem when the target space is a $\CAT(0)$ space (see \cite[Theorem 6.4]{GT2}).
\begin{thm}[\cite{GT2}]\label{thm:premain}
Let $(\X,\sfd,\mm)$ be a locally uniformly doubling, strongly rectifiable space satisfying a Poincar\'e inequality,
and let $\Omega$ be a bounded open subset of $\X$ with $\mm(\X\setminus \Omega)>0$.
Let $\Y_\o=(\Y,\sfd_\Y,\o)$ be a pointed $\CAT(0)$ space.
Let $\bar{u}\in \ks^{1,2}(\Omega,\Y_\o)$ be a Korevaar-Schoen type Sobolev map which determines a boundary value.
Then the Korevaar-Schoen type energy $\E^{\Omega}_{2,\bar{u}}:L^2(\Omega,\Y_\o)\to [0,+\infty]$ admits a unique minimizer.
\end{thm}

Thanks to Theorem \ref{thm:premain},
one can introduce the notion of harmonic map from an $\RCD(K,N)$ space to a $\CAT(0)$ space.
Very recently,
Gigli \cite{G} has shown a quantitative Lipschitz estimate for such harmonic maps,
and produced a Cheng type Liouville theorem (\cite{Che}) based on \cite{DGPS}, \cite{GN}, \cite{GT1} (see also \cite{ZZ}, \cite{ZZZ}).
Mondino-Semola \cite{MS} have also obtained a similar result, independently.

\subsection{$\CAT(1)$ target}\label{eq:cat1}

The purpose of this note is to yield an analogue of Theorem \ref{thm:premain} for the case where the target space is a regular ball in a $\CAT(1)$ space (i.e., geodesic ball whose radius is strictly less than $\pi/2$).
In the case where the source space is a Riemannian domain,
the solvability of the Dirichlet problem has been established by Serbinowski \cite{S}.
Similarly to the non-positively curved case,
the result in \cite{S} has been extended to non-smooth source spaces (see Eells-Fuglede \cite{EF}, Fuglede \cite{F2}, \cite{F3} for Riemannian polyhedra, and Huang-Zhang \cite{HZ} for Alexandrov spaces).
We now aim to generalize it for $\RCD(K,N)$ spaces.
Our main result is the following:
\begin{thm}\label{thm:main}
Let $(\X,\sfd,\mm)$ be an infinitesimally Hilbertian, locally uniformly doubling, strongly rectifiable space satisfying a Poincar\'e inequality,
and let $\Omega$ be a bounded open subset of $\X$ with $\mm(\X\setminus \Omega)>0$.
Let $\Y_{\o}=(\Y,\sfd_\Y,\o)$ be a pointed $\CAT(1)$ space,
and let $\bar{\B}_{\rho}(\o)$ be a regular ball $($i.e., a closed ball of radius $\rho \in (0,\pi/2)$ centered at $\o$$)$.
Let $\bar{u}\in \ks^{1,2}(\Omega,\bar{\B}_{\rho}(o))$ be a Korevaar-Schoen type Sobolev map which determines a boundary value.
Then the Korevaar-Schoen type energy $\E^{\Omega}_{2,\bar{u}}:L^2(\Omega,\bar{\B}_{\rho}(\o))\to [0,+\infty]$ admits a unique minimizer.
\end{thm}

Theorem \ref{thm:main} enables us to introduce the notion of harmonic map from an $\RCD(K,N)$ space into a regular ball in a $\CAT(1)$ space.
When the source space is a Riemannian domain,
Zhang-Zhong-Zhu \cite{ZZZ} have obtained a quantitative Lipschitz estimate and a Choi type Liouville theorem (\cite{Cho}).

\section{Preliminaries}\label{sec:Preliminaries}
We say that
$(\X,\sfd,\mm)$ is a metric measure space
if $(\X,\sfd)$ is a complete separable metric space,
and $\mm$ is a non-negative Borel measure,
which is finite on bounded sets.
This section is devoted to preliminaries for metric measure spaces.

\subsection{Sobolev functions}\label{eq:Sovolevfct}
We briefly recall the non-smooth differential calculus on metric measure spaces.
The readers can refer to \cite{G0}, \cite{GP} for the details.

Let $(\X,\sfd,\mm)$ be a metric measure space and let $C([0,1],\X)$ be the set of all curves in $\X$ defined on $[0,1]$ with the uniform topology.
For $t\in[0,1]$ the \textit{evaluation map} $\e_t:C([0,1],\X)\to \X$ is defined as $\e_t(\gamma):=\gamma_t$.
A curve $\gamma\in C([0,1],\X)$ is said to be \textit{absolutely continuous} if there is $f\in L^1(0,1)$ such that
\begin{equation}\label{eq:abs}
\sfd(\gamma_t,\gamma_s)\leq \int_s^t f(r)\,\d r
\end{equation}
for all $t,s\in[0,1]$ with $s< t$.
For an absolutely continuous $\gamma$,
the minimal $f$ satisfying \eqref{eq:abs} in the a.e. sense is called the \textit{metric speed} and it is denoted by $|\dot\gamma_t|$.
A Borel probability measure $\pi$ on $C([0,1],\X)$ is said to be a \textit{test plan} if
\begin{equation*}
\int_0^1\int|\dot\gamma_t|^2\,d\pi(\gamma)\,\d t<+\infty,\quad (\e_t)_{\#}\pi\leq C\,\mm
\end{equation*}
for some $C>0$. 
The \textit{Sobolev class $S^2(\X)$} is the set of all Borel functions $f:\X\to \mathbb{R}$ such that
there exists a non-negative $G\in L^2(\mm)$ such that
\begin{equation}\label{eq:sobclass}
\int |f(\gamma_1)-f(\gamma_0)|\,\d\pi(\gamma) \leq \int \int_0^1G(\gamma_t)|\dot\gamma_t|\,\d t\,\d\pi(\gamma)
\end{equation}
for all test plans $\pi$.
For $f\in S^2(\X)$,
a non-negative function $G\in L^2(\mm)$ satisfying \eqref{eq:sobclass} is called a \textit{weak upper gradient},
and the minimal one in the $\mm$-a.e. sense is called the \textit{minimal weak upper gradient} and it is denoted by $|Df|$.
The space
\begin{equation*}
W^{1,2}(\X):=L^2(\mm)\cap S^2(\X)
\end{equation*}
equipped with the norm
\begin{equation*}
\Vert f \Vert_{W^{1,2}(\X)}^2:=\Vert f\Vert^2_{L^2(\mm)}+\Vert |Df| \Vert^2_{L^2(\mm)}
\end{equation*}
is called \textit{Sobolev space} and it can be proved that it is a Banach space.
%

It is well-known that
there exists a unique couple $(L^2(T^{*}\X),\d)$,
where $L^2(T^{*}\X)$ is an $L^2(\mm)$-normed $L^{\infty}(\mm)$-module and $\d:S^2(\X)\to L^2(T^{*}\X)$ is a linear operator such that the following hold (see e.g., \cite{G0}, \cite[Theorem 4.1.1]{GP}):
\begin{enumerate}\setlength{\itemsep}{+0.7mm}
\item $L^2(T^{*}\X)$ is generated by $\{\d f \mid  f\in S^2(\X)\}$;
\item for every $f\in S^2(\X)$, it holds that
\begin{equation*}
|\d f|=|Df|\quad \text{$\mm$-a.e.}.
\end{equation*}
\end{enumerate}
The space $L^2(T^{*}\X)$ and the operator $\d$ are called the \textit{cotangent module} and the \textit{differential},
respectively.
The \textit{tangent module} $L^2(T\X)$ is defined as the dual module of $L^2(T^{*}\X)$.
For $f\in S^2(\X)$,
let $\Grad(f)$ be the set of all $v\in L^2(T\X)$ such that
\begin{equation*}
\d f(v)=|\d f|^2=|v|^2\quad \text{$\mm$-a.e.}.
\end{equation*}
Note that $\Grad(f)$ is non-empty (see \cite[Remark 4.2.10]{GP}).
The following chain rule holds (see e.g., \cite[Theorem 4.2.15]{GP}):
\begin{thm}[\cite{GP}]\label{thm:gradchain}
Let $f \in S^2(\X)$ and $v\in \Grad(f)$.
If $\varphi:\mathbb{R}\to \mathbb{R}$ is Lipschitz,
then $\varphi' \circ f v\in \Grad (\varphi \circ f)$,
where $\varphi'\circ f$ is arbitrarily defined on the inverse image of the non-differentiability points of $\varphi$.
\end{thm}

A metric measure space $(\X,\sfd,\mm)$ is said to be \textit{infinitesimally strictly convex} when for every $f\in S^2(\X)$ the set $\Grad(f)$ is a singleton;
in this case,
the unique element is denoted by $\nabla f$.
Also,
$(\X,\sfd,\mm)$ is said to be \textit{infinitesimally Hilbertian} if $W^{1,2}(\X)$ is a Hilbert space.

We have the following characterization (see e.g., \cite[Theorem 4.3.3]{GP}):
\begin{thm}[\cite{GP}]\label{thm:IH}
The following are equivalent:
\begin{enumerate}\setlength{\itemsep}{+0.7mm}
\item $(\X,\sfd,\mm)$ is infinitesimally Hilbertian;
\item $L^2(T^{*}\X)$ and $L^2(T\X)$ are Hilbert modules;
\item $(\X,\sfd,\mm)$ is infinitesimally strictly convex,
and for all $f,g\in W^{1,2}(\X)$,
\begin{equation*}
\nabla (f+g)=\nabla f+\nabla g \quad \text{$\mm$-a.e.};
\end{equation*}
\item $(\X,\sfd,\mm)$ is infinitesimally strictly convex,
and for all $f,g\in W^{1,2}(\X) \cap L^{\infty}(\mm)$,
\begin{equation*}
\nabla (f g)=f \nabla g+g \nabla f \quad \text{$\mm$-a.e.}.
\end{equation*}
\end{enumerate}
\end{thm}

In the case where $(\X,\sfd,\mm)$ is infinitesimally Hilbertian,
we denote by $\langle \cdot,\cdot \rangle$ the pointwise scalar product on $L^2(T\X)$.

\subsection{Doubling, Poincar\'e and strong rectifiability}\label{sec:PI}
Here we review a volume doubling property, a Poincar\'e inequality and the strongly rectifiability on metric measure spaces.

Let $(\X,\sfd,\mm)$ be a metric measure space.
For $r>0$ and $x\in \X$,
we denote by $B_r(x)$ the open ball of radius $r$ centered at $x$.
A metric measure space $(\X,\sfd,\mm)$ is said to be \textit{locally uniformly doubling} if for every $R>0$ there exists $C_R>0$ such that
\begin{equation*}
\mm(B_{2r}(x)) \leq C_R \,\mm(B_r(x))
\end{equation*}
for all $x\in \X$ and $r\in(0,R)$.
We also say that
$(\X,\sfd,\mm)$ satisfies a \textit{Poincar\'e inequality} if for every $R>0$ there are $C_R,\lambda=\lambda_R>0$ such that
\begin{equation*}
\fint_{B_r(x)}|f-f_{B_r(x)}|\,\d\mm \leq C_R\,r\fint_{ B_{\lambda r}(x)} \lip (f)\, \d\mm
\end{equation*}
holds for every Lipschitz function $f:\X\to \mathbb{R}$ and for all $x\in \X$ and $r\in(0,R)$,
where $f_{B_r(x)}:=\fint_{B_r(x)} f\,\d\mm$ and $\lip(f)$ is the pointwise Lipschitz constant.

We say that
a metric measure space $(\X,\sfd,\mm)$ is a \textit{$d$-dimensional strongly rectifiable space} if
for every $\eps>0$ there exists an \textit{$\eps$-atlas} $\mathcal A^\eps:=\{(U^\eps_i,\varphi^\eps_i)\}_{i}$ such that the following hold:
\begin{enumerate}\setlength{\itemsep}{+0.7mm}
\item Borel subsets $\{U^\eps_i\}_i$ forms a partition of $\X$ up to an $\mm$-negligible set;
\item $\varphi^\eps_i$ is a $(1+\eps)$-bi-Lipschitz map from $U^\eps_i$ to $\varphi^\eps_i(U^\eps_i)\subset \mathbb{R}^d$;
\item for some $c_i>0$, we have
\begin{equation*}
c_i\,\mathcal L^d|_{\varphi^\eps_i(U^\eps_i)}\leq (\varphi^\eps_i)_{\#}\mm|_{U^\eps_i}\leq (1+\eps)c_i\,\mathcal L^d|_{\varphi^\eps_i(U^\eps_i)}.
\end{equation*}
\end{enumerate}

\begin{rem}
For $K\in \mathbb{R}$ and $N\in [1,\infty)$,
every $\RCD(K,N)$ space is a locally uniformly doubling, strongly rectifiable space satisfying a Poincar\'e inequality.
Actually,
the doubling property is due to \cite{LV}, \cite{St1}, \cite{St2},
the Poincar\'e inequality is due to \cite{R}.
Moreover,
as stated in \cite[Theorem 2.19]{GT2},
the strong rectifiability is a consequence of the results by Bru\'{e}-Semola \cite{BS}, Gigli-Pasqualetto \cite{GP2}, Kell-Mondino \cite{KM}, Mondino-Naber \cite{MN}, De Philippis-Marchese-Rindler \cite{DMR}.
\end{rem}

\begin{rem}
Let $(\X,\sfd,\mm)$ be a strongly rectifiable space,
and let $\{\eps_n\}_n$ be a sequence with $\eps_n\downarrow0$.
A family $\{\mathcal A^{\eps_n}\}_{n}$ of atlases is said to be \textit{aligned} if
for all $n,m$ and $(U^{\eps_n}_i,\varphi^{\eps_n}_i)\in\mathcal A^{\eps_n}$, $(U^{\eps_m}_j,\varphi^{\eps_m}_j)\in\mathcal A^{\eps_m}$,
the map $\varphi^{\eps_n}_i-\varphi^{\eps_m}_j$ is $(\eps_n+\eps_m)$-Lipschitz on $U^{\eps_n}_i\cap U^{\eps_m}_j$.
It has been observed in \cite[Theorem 3.9]{GP3}, \cite[Subsection 2.4]{GT2} that
for any sequence $\{\eps_n\}_n$,
an aligned family of atlases $\{\mathcal A^{\eps_n}\}_{n}$ exists.
\end{rem}

\subsection{Approximate metric differentiability}\label{sec:approx}
We further review the notion of approximate metric differentiability of maps introduced by Kirchheim \cite{K} and Gigli-Tyulenev \cite{GT2}.

Let $(\X,\sfd,\mm)$ be a metric measure space,
and let $(\Y,\sfd_\Y)$ be a metric space.
Let $u:\X\to \Y$ be a Borel map.
The \textit{pointwise Lipschitz constant of $u$ at $x\in \X$} is defined as
\begin{equation*}
\lip (u)(x):=\varlimsup_{y\to x}\frac{\sfd_\Y(u(x),u(y))}{\sfd(x,y)}
\end{equation*}
if $x$ is not an isolated point,
and $\lip (u)(x):=0$ otherwise.
For a Borel subset $U\subset \X$,
a point $x\in U$ is said to be a \textit{density point of $U$} if
\begin{equation*}
\lim_{r\downarrow 0} \frac{\mm(B_r(x)\cap U)}{\mm(B_r(x))}=1.
\end{equation*}
The \textit{approximate Lipschitz constant} is defined by
\begin{equation*}
\aplip(u)(x):=\aplims_{y\to x}\frac{\sfd_\Y(u(x),u(y))}{\sfd(x,y)},
\end{equation*}
where the right hand side means the approximate upper limit (see e.g., \cite[Subsection 2.1]{GT2} for the precise definition).
We have the following (see \cite[Proposition 2.5]{GT2}):
\begin{prop}[\cite{GT2}]\label{prop:loclip}
Let $(\X,\sfd,\mm)$ be a uniformly locally doubling space,
and let $(\Y,\sfd_\Y)$ be a complete metric space.
Let $U$ be a Borel subset of $\X$,
and let $u:U\to \X$ be Lipschitz.
Then for every density point $x\in \X$ of $U$ we have
\begin{equation*}
\lip (u)(x)=\aplip(u)(x).
\end{equation*}
\end{prop}

Let $(\X,\sfd,\mm)$ be a $d$-dimensional strongly rectifiable space,
and let $\{\mathcal A^{\eps_n}\}_n$ be an aligned family of atlases for a sequence $\{\eps_n\}_n$ with $\eps_n\downarrow0$.
Let $(\Y,\sfd_\Y)$ be a metric space.
We denote by $\sn^d$ the set of all semi-norms on $\mathbb{R}^d$ equipped with the complete separable metric
\begin{equation*}
\D(\nn_1,\nn_2):=\sup_{z}|\nn_1(z)-\nn_2(z)|,
\end{equation*}
where the supremum is taken over all $z\in \mathbb{R}^d$ with $|z|\leq 1$.
We write $\Vert \nn \Vert:=\D(\nn,0)$.
We say that
a Borel map $u:\X\to\Y$ is \textit{approximately metrically differentiable at $x\in\X$ relatively to $\{\mathcal A^{\eps_n}\}_n$} if the following hold:
\begin{enumerate}\setlength{\itemsep}{+0.7mm}
\item For every $n$, there exists $i=i_{x,n}$ such that
$x$ belongs to $U_{i}^{\eps_n}$, it is a density point of $U_{i}^{\eps_n}$ and $\varphi^{\eps_n}_{i}(x)$ is a density point of $\varphi^{\eps_n}_{i}(U^{\eps_n}_{i})$;
\item there exists $\md_x(u)\in \sn^d$, called the \textit{metric differential of $u$ at $x$}, such that
\begin{equation*}
\lims_{n\to\infty}\aplims_{y\to x,\,y\in U_{i}^{\eps_n}}\frac{\big|\sfd_\Y(u(y),u(x))-\md_x(u)(\varphi^{\eps_n}_{i}(y)-\varphi^{\eps_n}_{i}(x))\big|}{\sfd(y,x)}=0.
\end{equation*}
\end{enumerate}

Moreover,
Gigli-Tyulenev \cite{GT2} have shown the following (see \cite[Lemma 3.4]{GT2}):
\begin{lem}[\cite{GT2}]\label{lem:approx}
Let $(\X,\sfd,\mm)$ be a strongly rectifiable space,
and let $\{\mathcal A^{\eps_n}\}_n$ be an aligned family of atlases for a sequence $\{\eps_n\}_n$ with $\eps_n\downarrow0$.
Let $(\Y,\sfd_\Y)$ be a complete metric space.
If $u:\X\to\Y$ is approximately metrically differentiable at $x\in\X$ relatively to $\{\mathcal A^{\eps_n}\}_n$,
then
\begin{equation*}
\aplip(u)(x)=\Vert \md_x(u) \Vert.
\end{equation*}
\end{lem}

Gigli-Tyulenev \cite{GT2} have further concluded the $\mm$-a.e. approximately metrical differentiability of Borel maps satisfying the Lusin-Lipschitz property (see \cite[Proposition 3.6]{GT2}).


\subsection{Korevaar-Schoen space}\label{sec:KS}
In this section
we recall the formulation and basic results concerning the Korevaar-Schoen type energy introduced in Gigli-Tyulenev \cite{GT2}.
Let $(\X,\sfd,\mm)$ be a metric measure space,
and let $\Omega$ be an open subset of $\X$.
Let $\Y_\o=(\Y,\sfd_\Y,\o)$ be a pointed complete metric space.

We denote by $L^0(\Omega,\Y)$ the set of all Borel maps (up to $\mm$-a.e. equality) from $\Omega$ to $\Y$ with separable range.
Let $L^2(\Omega,\Y_\o)$ be the set of all $u\in L^0(\Omega,\Y)$ such that
\begin{equation*}
\int_{\Omega} \sfd_{\Y}^2(u(x),\o)\,\d\mm(x)<+\infty,
\end{equation*}
which is endowed with the metric
\begin{equation*}
\sfd_{L^2}(u,v):=\left| \int_{\Omega} \sfd_\Y^2\big(u(x),v(x)\big)\,\d\mm(x)\right|^{1/2}.
\end{equation*}
Since $\Y$ is complete,
it can be proved that so is $L^2(\Omega,\Y_{\o})$.

Let $u\in L^2(\Omega,\Y_{\o})$.
For $r>0$ and $x\in \Omega$,
the \textit{approximate energy} is defined by
\begin{equation*}
\kse_{2,r}[u,\Omega](x):= \begin{cases}
                                  \displaystyle{\left|\fint_{B_r(x)} \frac{\sfd^2_\Y(u(x),u(y))}{r^2}\,\d \mm(y) \right|^{1/2}} & \textrm{if}\ B_{r}(x)\subset \Omega,\\
                                  0 & \textrm{otherwise}.
                                  \end{cases}
\end{equation*}
The \textit{Korevaar-Schoen space $\ks^{1,2}(\Omega,\Y_\o)$} is defined by the set of all $u\in L^2(\Omega,\Y_\o)$ such that
\begin{equation}\label{eq:defkso}
\E_2^\Omega(u):=\limi_{r\downarrow0}\,\int_\Omega \,\kse^2_{2,r}[u,\Omega]\,\d\mm<+\infty.
\end{equation}

\begin{rem}
Let $W^{1,2}(\Omega)$ denote the Sobolev space provided by means of \cite[Definition 5.2]{GT2}.
One can also introduce the associated Sobolev space $W^{1,2}(\Omega,\Y_{\o})$ as the set of all $u\in L^2(\Omega,\Y_{\o})$ such that
there is $G\in L^2(\Omega)$ such that for all $1$-Lipschitz functions $f:\Y\to\mathbb{R}$ we have $f\circ u\in W^{1,2}(\Omega)$ and $|D(f\circ u)|\leq G$ $\mm$-a.e.\ on $\Omega$ (see \cite[Definition 5.3]{GT2}). 
\end{rem}

We now present some basic properties of the Korevaar-Schoen space $\ks^{1,2}(\Omega,\Y_\o)$ (see \cite[Theorem 5.7]{GT2}):
\begin{thm}[\cite{GT2}]\label{thm:openenerygy}
Let $(\X,\sfd,\mm)$ be a  locally uniformly doubling, strongly rectifiable space satisfying a Poincar\'e inequality,
and let $\Omega$ be an open subset.
Let $\Y_\o=(\Y,\sfd_\Y,\o)$ be a pointed complete metric space.
Then the following hold:
\begin{enumerate}\setlength{\itemsep}{+0.7mm}
\item $\ks^{1,2}(\Omega,\Y_\o)=W^{1,2}(\Omega,\Y_\o)$;
\item for every $u\in \ks^{1,2}(\Omega,\Y_{\o})$
there exists a function $\e_2[u]\in L^2(\Omega)$, called the \textit{energy density}, such that
\begin{equation*}
\kse_{2,r}[u,\Omega] \to  \e_2[u]\quad\text{ $\mm$-a.e. on $\Omega$ and in $L^2(\Omega)$ as $r\downarrow0$}.
\end{equation*}
In particular $\liminf$ in \eqref{eq:defkso} is a limit and $\E_2^\Omega(u)$ can be written as
\begin{equation*}
\E_2^\Omega(u)= \begin{cases}
                                  \displaystyle{\int_\Omega \e_2^2[u]\, \d\mm} & \textrm{if}\ u\in \ks^{1,2}(\Omega,\Y_{\o}),\\
                                  +\infty & \textrm{otherwise};
                                  \end{cases}
\end{equation*}
\item $\E^\Omega_2:L^2(\Omega,\Y_{\o})\to[0,+\infty]$ is lower semicontinuous;
\item any $u\in \ks^{1,2}(\Omega,\Y_\o)$ is approximately metrically differentiable $\mm$-a.e. in $\Omega$,
where we extend $u$ to the whole $\X$ by setting it to be constant outside of $\Omega$.
\end{enumerate}
\end{thm}

We also need the following (see \cite[Propositions 4.6 and 4.19]{GT2}):
\begin{prop}[\cite{GT2}]\label{thm:consR}
Let $(\X,\sfd,\mm)$ be a $d$-dimensional, locally uniformly doubling, strongly rectifiable space satisfying a Poincar\'e inequality,
and let $\Omega$ be an open subset.
Then there exists $c_d>0$ depending only on $d$ such that
for every $u\in \ks^{1,2}(\Omega,\mathbb{R})$,
\begin{equation}\label{eq:consR}
\e_2[u]=c_d\,|Du| \quad \text{$\mm$-a.e. on $\Omega$}.
\end{equation}
\end{prop}


Let $W^{1,2}_0(\Omega)$ be the $W^{1,2}(\Omega)$-closure of the set of all functions in $W^{1,2}(\Omega)$ whose support is contained in $\Omega$.
For a fixed $\bar u\in \ks^{1,2}(\Omega,\Y_{\o})$,
we define
\begin{equation*}
\ks^{1,2}_{\bar u}(\Omega,\Y_{\o}):=\big\{u\in\ks^{1,2}(\Omega,\Y_{\o}) \mid \sfd_\Y(u,\bar u)\in W^{1,2}_0(\Omega)\big\},
\end{equation*}
and the associated energy functional $\E^\Omega_{2,\bar u}:L^2(\Omega,\Y_\o)\to[0,+\infty]$ as
\begin{equation*}
\E^\Omega_{2,\bar u}(u):= \begin{cases}
                                  \displaystyle{\E_2^\Omega(u)} & \textrm{if}\ u\in \ks^{1,2}_{\bar u}(\Omega,\Y_{\o}),\\
                                  +\infty & \textrm{otherwise}.
                                  \end{cases}
\end{equation*}

We close this section with the following (see \cite[Proposition 5.10]{GT2}):
\begin{prop}[\cite{GT2}]\label{prop:Dirichlet}
Let $(\X,\sfd,\mm)$ be a locally uniformly doubling, strongly rectifiable space satisfying a Poincar\'e inequality,
and let $\Omega$ be an open subset.
Let $\Y_\o=(\Y,\sfd_\Y,\o)$ be a pointed complete metric space.
Then for a fixed $\bar u\in \ks^{1,2}(\Omega,\Y_{\o})$ the following hold:
\begin{enumerate}\setlength{\itemsep}{+0.7mm}
\item $\E^\Omega_{2,\bar u}$ is lower semicontinuous;
\item for any $u,v\in \ks^{1,2}_{\bar u}(\Omega,\Y_\o)$ we have $\sfd_\Y(u,v)\in W^{1,2}_0(\Omega)$.
\end{enumerate}
\end{prop}


\section{Dirichlet problem}\label{sec:Dirichlet problem}
In the present section,
we give a proof of Theorem \ref{thm:main}.
We proceed along the line of the proof of \cite[Theorem 1.16]{S},  \cite[Theorem 6.4]{GT2}.

\subsection{Technical lemma}\label{sec:tec}
This subsection is devoted to the proof of a technical lemma.
Let $(\X,\sfd,\mm)$ be a metric measure space,
and let $\Omega$ be an open subset of $\X$.
Let $\Y_\o=(\Y,\sfd_\Y,\o)$ be a pointed complete metric space.
Let $u,v,w \in L^2(\Omega,\Y_{\o})$.
For $\alpha>0$ and $x\in \Omega$,
we define the following (Borel) sets
\begin{align}\label{eq:tecset}
\Omega_{v,\alpha,x}&:=\{y\in \Omega \mid \sfd_\Y(v(x),v(y))\geq \alpha\},\\ \notag
\Omega_{w,\alpha,x}&:=\{y\in \Omega \mid \sfd_\Y(w(x),w(y))\geq \alpha\},\\ \notag
\Omega_{v,w,\alpha,x}&:=\Omega_{v,\alpha,x}\cup \Omega_{w,\alpha,x}.
\end{align}
We also define a \textit{modified approximate energy} by
\begin{equation}\label{eq:tecfct}
\kse_{2,r}[u,\Omega;v,w,\alpha](x):= \begin{cases}
                                  \displaystyle{\left|\frac{1}{\mm(B_r(x))}\int_{B_r(x)\setminus \Omega_{v,w,\alpha,x}} \frac{\sfd^2_\Y(u(x),u(y))}{r^2}\,\d \mm(y) \right|^{1/2}} & \textrm{if}\ B_{r}(x)\subset \Omega,\\
                                  0 & \textrm{otherwise}.
                                  \end{cases}
\end{equation}

We now state and prove our key lemma (cf. \cite[Lemma 1.12]{S}, \cite[Lemma 3.1]{HZ}):
\begin{lem}\label{lem:tech}
Let $(\X,\sfd,\mm)$ be a  locally uniformly doubling, strongly rectifiable space satisfying a Poincar\'e inequality,
and let $\Omega$ be an open subset.
Let $\Y_\o=(\Y,\sfd_\Y,\o)$ be a pointed complete metric space.
Let $u,v,w \in \ks^{1,2}(\Omega,\Y_\o)$ and $\alpha>0$ fixed.
Then we have
\begin{equation}\label{eq:tech}
\kse_{2,r}[u,\Omega;v,w,\alpha] \to  \e_2[u]\quad\text{$\mm$-a.e. on $\Omega$ and in $L^2(\Omega)$ as $r\downarrow0$}.
\end{equation}
\end{lem}
\begin{proof}
We first prove the $\mm$-a.e. convergence.
Note that
for $\mm$-a.e. $x\in \Omega$ the following hold:
\begin{enumerate}\setlength{\itemsep}{+0.7mm}
\item $\kse_{2,r}[u,\Omega](x)\to \e_2[u](x), \kse_{2,r}[v,\Omega](x)\to \e_2[v](x),\kse_{2,r}[w,\Omega](x)\to \e_2[w](x)$ as $r\downarrow0$;
\item $\lip u(x)<+\infty$.
\end{enumerate}
Indeed,
the first property is a consequence of Theorem \ref{thm:openenerygy}.
Concerning the second one,
Theorem \ref{thm:openenerygy} implies that
$u$ is approximately metrically differentiable $\mm$-a.e. $x\in \Omega$ relatively to an aligned family $\{\mathcal A^{\eps_n}\}_n$ of atlases.
For such an $x\in \Omega$,
and for each $n$, there is $i=i_{x,n}$ such that
$x$ belongs to a chart $U_{i}^{\eps_n}$ of $\mathcal A^{\eps_n}$, 
and it is a density point of $U_{i}^{\eps_n}$.
Thanks to \cite[(2.15), Corollary 3.10]{GT2},
we may assume that $u|_{U_{i}^{\eps_n}}$ is Lipschitz.
By Proposition \ref{prop:loclip} and Lemma \ref{lem:approx},
$\lip u(x)$ coincides with $\Vert \md_x(u)\Vert$,
which is finite.

Let $x\in \Omega$ be such that two properties above hold:
we shall prove \eqref{eq:tech} for such an $x$.
Since
\begin{equation*}
\kse^2_{2,r}[u,\Omega;v,w,\alpha](x)=\kse^2_{2,r}[u,\Omega](x)-\frac{1}{\mm(B_r(x))}   \int_{B_r(x)\cap \Omega_{v,w,\alpha,x}} \frac{\sfd^2_\Y(u(x),u(y))}{r^2}\,\d \mm(y),
\end{equation*}
it suffices to show that
\begin{equation}\label{eq:techkey}
\frac{1}{\mm(B_r(x))}   \int_{B_r(x)\cap \Omega_{v,w,\alpha,x}} \frac{\sfd^2_\Y(u(x),u(y))}{r^2}\,\d \mm(y)\to 0
\end{equation}
as $r\downarrow0$.
We have 
\begin{align*}\notag
&\quad \,\,\frac{1}{\mm(B_r(x))}   \int_{B_r(x)\cap \Omega_{v,w,\alpha,x}} \frac{\sfd^2_\Y(u(x),u(y))}{r^2}\,\d \mm(y)\\ \notag
&\leq \frac{\mm(B_r(x)\cap \Omega_{v,w,\alpha,x})}{\mm(B_r(x))} \sup_{y\in B_r(x)\setminus \{x\}}\frac{\sfd^2_{\Y}(u(x),u(y))}{\sfd^2(x,y)}\\
&\leq \left(\frac{\mm(B_r(x)\cap \Omega_{v,\alpha,x})}{\mm(B_r(x))}+\frac{\mm(B_r(x)\cap \Omega_{w,\alpha,x})}{\mm(B_r(x))} \right)\sup_{y\in B_r(x)\setminus \{x\}}\frac{\sfd^2_{\Y}(u(x),u(y))}{\sfd^2(x,y)}\\ \label{eq:techkey2}
&\leq \frac{r^2}{\alpha^2} \left( \kse^2_{2,r}[v,\Omega](x)+\kse^2_{2,r}[w,\Omega](x) \right)\sup_{y\in B_r(x)\setminus \{x\}}\frac{\sfd^2_{\Y}(u(x),u(y))}{\sfd^2(x,y)},
\end{align*}
where the last inequality follows from
\begin{align*}
\frac{\alpha^2}{r^2} \frac{\mm(B_r(x)\cap \Omega_{v,\alpha,x})}{\mm(B_r(x))}&\leq \frac{1}{\mm(B_r(x))} \int_{B_r(x)\cap \Omega_{v,\alpha,x}} \frac{\sfd^2_\Y(v(x),v(y))}{r^2}\,\d \mm(y)\leq \kse^2_{2,r}[v,\Omega](x),\\
\frac{\alpha^2}{r^2} \frac{\mm(B_r(x)\cap \Omega_{w,\alpha,x})}{\mm(B_r(x))}&\leq \frac{1}{\mm(B_r(x))} \int_{B_r(x)\cap \Omega_{w,\alpha,x}} \frac{\sfd^2_\Y(w(x),w(y))}{r^2}\,\d \mm(y)\leq \kse^2_{2,r}[w,\Omega](x).
\end{align*}
Therefore,
by letting $r\downarrow0$ we obtain \eqref{eq:techkey} and the sought $\mm$-a.e. convergence.
The $L^2$-convergence follows from $\kse^2_{2,r}[u,\Omega;v,w,\alpha]\leq \kse^2_{2,r}[u,\Omega]$ and \cite[Lemma 3.14]{GT2}.
\end{proof}

\subsection{$\CAT(\kappa)$ spaces}\label{sec:CAT}
In this subsection we shall review some basic properties of $\CAT(\kappa)$ spaces.
We refer to \cite{BBI}, \cite{BH}.
For $\kappa\in \mathbb{R}$,
let $M_\kappa$ be the $2$-dimensional space form of constant curvature $\kappa$.
Let $D_{\kappa}$ be the diameter of $M_{\kappa}$;
more precisely,
\begin{equation*}
D_\kappa:= \begin{cases}
                                  \pi/\sqrt{\kappa} & \textrm{if}\ \kappa>0,\\
                                  +\infty & \textrm{otherwise}.
                                  \end{cases}
\end{equation*}

Let $(\Y,\sfd_\Y)$ be a metric space.
A curve $\gamma:[0,1]\to \Y$ is said to be a \textit{geodesic} if
\begin{equation*}
\sfd_\Y(\gamma_t,\gamma_s)=|s-t|\sfd_\Y(\gamma_0,\gamma_1)
\end{equation*}
for all $t,s\in [0,1]$.
Further,
$(\Y,\sfd)$ is called a \textit{geodesic space} if for any pair of points $x,y\in \Y$,
there is a geodesic $\gamma:[0,1]\to \Y$ with $\gamma_0=x$ and $\gamma_1=y$.
For $\kappa \in \mathbb{R}$,
a complete geodesic space $(\Y,\sfd_\Y)$ is said to be a \textit{$\CAT(\kappa)$-space} if it satisfies the following \textit{$\kappa$-triangle comparison principle}:
for every geodesic triangle $\triangle_{pqr}$ in $\Y$ whose perimeter is less than $2 D_{\kappa}$,
and for every point $x$ on the segment between $q$ and $r$,
it holds that
\begin{equation*}
\sfd_{\Y}(p,x)\leq \sfd_{\kappa}(\bar{p},\bar{x}),
\end{equation*}
where $\sfd_{\kappa}$ is the Riemannian distance on $M_\kappa$,
and $\bar{p}$ and $\bar{x}$ are comparison points in $M_\kappa$.

Let us recall the following basic properties (see e.g., \cite[Proposition 1.4]{BH}):
\begin{prop}[\cite{BH}]\label{prop:cat}
For $\kappa \in \mathbb{R}$,
let $(\Y,\sfd_\Y)$ be a $\CAT(\kappa)$ space.
Then we have:
\begin{enumerate}\setlength{\itemsep}{+0.7mm}
\item  For each pair of points $x,y\in \Y$ with $\sfd_\Y(x,y)<D_{\kappa}$,
there exists a unique geodesic from $x$ to $y$;
moreover,
it varies continuously with its end points;
\item any ball in $\Y$ whose radius is less than $D_{\kappa}/2$ is convex;
namely,
for each pair of points in such a ball,
the unique geodesic joining them is contained the ball. 
\end{enumerate}
\end{prop}

\begin{rem}
In \cite{BH},
Proposition \ref{prop:cat} has been deduced from the following:
let $\gamma,\eta:[0,1]\to \Y$ be geodesics such that
\begin{equation*}
\gamma_0=\eta_0,\quad \sfd_\Y(\gamma_0,\gamma_1)+\sfd_\Y(\gamma_1,\eta_1)+\sfd_\Y(\eta_1,\eta_0)<2D_{\kappa}.
\end{equation*}
For $l\in (0,D_{\kappa})$,
we assume $\sfd_\Y(\gamma_0,\gamma_1)\leq l$ and $\sfd_\Y(\eta_0,\eta_1)\leq l$.
Then there is a constant $c_{\kappa,l}>0$ depending only on $\kappa,l$ such that
for all $t\in [0,1]$,
\begin{equation}\label{eq:rough3}
\sfd_\Y(\gamma_t,\eta_t)\leq c_{\kappa,l}\,\sfd_\Y(\gamma_1,\eta_1).
\end{equation}
\end{rem}

We now collect some useful estimates on $\CAT(1)$ spaces,
which have been stated in \cite{S} without proof.
The detailed proofs can be found in \cite[Appendix A]{BFHMSZ}.
The first one is the following (see \cite[p. 11, ESTIMATE I]{S}, and also \cite[Lemma A.2]{BFHMSZ}):
\begin{lem}[\cite{S}]\label{lem:cos1}
Let $(\Y,\sfd_\Y)$ be a $\CAT(1)$ space,
and let $\gamma,\eta:[0,1]\to \Y$ be geodesics with
\begin{equation*}
\sfd_\Y(\gamma_0,\eta_0)+\sfd_\Y(\eta_0,\eta_1)+\sfd_\Y(\eta_1,\gamma_1)+\sfd_\Y(\gamma_1,\gamma_0)<2\pi.
\end{equation*}
Then we have
\begin{align*}
\cos^2 \frac{\sfd_\Y(\gamma_0,\gamma_1)}{2}&\,\sfd^2_\Y(\gamma_{1/2},\eta_{1/2})+\frac{1}{4}(\sfd_\Y(\eta_0,\eta_1)-\sfd_\Y(\gamma_0,\gamma_1))^2\\
&\leq \frac{1}{2}\left( \sfd^2_\Y(\gamma_0,\eta_0)+\sfd^2_\Y(\gamma_1,\eta_1)  \right)\\
&+\Cub(\sfd_\Y(\gamma_0,\eta_0),\sfd_\Y(\gamma_1,\eta_1),\sfd_\Y(\eta_0,\eta_1)-\sfd_\Y(\gamma_0,\gamma_1),\sfd_\Y(\gamma_{1/2},\eta_{1/2})),
\end{align*}
where $\Cub$ denotes cubic terms in the indicated variables.
\end{lem}

The second one is the following (see \cite[p. 13, ESTIMATE II]{S}, and also \cite[Lemma A.4]{BFHMSZ}):
\begin{lem}[\cite{S}]\label{lem:cos2}
Let $(\Y,\sfd_\Y)$ be a $\CAT(1)$ space,
and let $\gamma,\eta:[0,1]\to \Y$ be geodesics with $\gamma_1=\eta_1$ and
\begin{equation*}
\sfd_\Y(\gamma_0,\eta_0)+\sfd_\Y(\eta_0,\eta_1)+\sfd_\Y(\eta_1,\gamma_0)<2\pi.
\end{equation*}
Then for every $t,s\in [0,1]$ we have
\begin{align*}
\sfd^2_\Y(\gamma_t,\eta_s)&\leq \frac{\sin^2(1-t)\sfd_\Y(\gamma_0,\gamma_1)}{\sin^2 \sfd_\Y(\gamma_0,\gamma_1) }\left(\sfd^2_{\Y}(\gamma_0,\eta_0)-(\sfd_\Y(\eta_0,\eta_1)-\sfd_\Y(\gamma_0,\gamma_1))^2 \right)\\
&+(1-t)^2(\sfd_\Y(\eta_0,\eta_1)-\sfd_\Y(\gamma_0,\gamma_1))^2+\sfd^2_\Y(\gamma_0,\gamma_1)(s-t)^2\\
&-2(1-t)(s-t)\sfd_\Y(\gamma_0,\gamma_1)(\sfd_\Y(\eta_0,\eta_1)-\sfd_\Y(\gamma_0,\gamma_1))\\
&+\Cub(s-t,\sfd_{\Y}(\gamma_0,\eta_0),\sfd_\Y(\eta_0,\eta_1)-\sfd_\Y(\gamma_0,\gamma_1),\sfd_\Y(\gamma_t,\eta_s)).
\end{align*}
\end{lem}

\begin{rem}\label{rem:error}
In Lemmas \ref{lem:cos1} and \ref{lem:cos2},
we further see that
in each term of $\Cub$,
at least one of the variables has an exponent which is greater than or equal to two. 
\end{rem}

\subsection{Energy estimates}\label{sec:key}

In this subsection
we are going to obtain some energy estimates which will be exploited in the proof of our main theorem.
Let $(\X,\d,\mm)$ be a locally uniformly doubling, strongly rectifiable space satisfying a Poincar\'e inequality,
and let $\Omega$ be an open subset of $X$.
Let $\Y_\o=(\Y,\sfd_\Y,\o)$ be a pointed $\CAT(1)$ space.
For $\rho \in (0,\pi/2)$,
we denote by $\bar{\B}_{\rho}(\o)$ the closed ball in $\Y$ of radius $\rho$ centered at $\o$.
In view of Proposition \ref{prop:cat},
$\bar{\B}_{\rho}(\o)=(\bar{\B}_{\rho}(\o),\sfd_\Y,\o)$ is a pointed $\CAT(1)$ space with itself.
We call $\bar{\B}_{\rho}(\o)$ a \textit{regular ball}.

For $u,v\in L^0(\Omega,\bar{\B}_{\rho}(\o))$ and $t\in [0,1]$,
we define a map $\G^{u,v}_t:\Omega \to \bar{\B}_{r}(\o)$ by $x\mapsto \G^{u(x),v(x)}_t$,
where $\G^{u(x),v(x)}$ denotes the unique geodesic from $u(x)$ to $v(x)$.
By virtue of Proposition \ref{prop:cat},
$\G^{u,v}_t$ belongs to $L^0(\Omega,\bar{\B}_{\rho}(\o))$.
By the same argument as in \cite[Section 6]{GT2},
we see that if $u,v\in  L^2(\Omega,\bar{\B}_{\rho}(\o))$, then $\G^{u,v}_t\in L^2(\Omega,\bar{\B}_{\rho}(\o))$,
and it is the unique geodesic from $u$ to $v$.

%
%

We begin with the following energy density estimate (cf. \cite[Lemma 1.13]{S}):
\begin{lem}\label{lem:energy1}
Let $u,v\in \ks^{1,2}(\Omega,\bar{\B}_{\rho}(\o)), m:=\G^{u,v}_{1/2}$ and $\sfd:=\sfd_\Y(u,v)$.
Then we have $m \in \ks^{1,2}(\Omega,\bar{\B}_{\rho}(\o)), \sfd \in \ks^{1,2}(\Omega,\mathbb{R})$
and
\begin{equation}\label{eq:energy12}
\cos^2 \frac{\sfd}{2}\, \e^2_2[m]+\frac{1}{4}\e^2_2[\sfd]\leq \frac{1}{2}\left( \e^2_2[u]+\e^2_{2}[v] \right) \quad \text{$\mm$-a.e. on $\Omega$}.
\end{equation}
\end{lem}
\begin{proof}
We check $m \in \ks^{1,2}(\Omega,\bar{\B}_{\rho}(\o)), \sfd \in \ks^{1,2}(\Omega,\mathbb{R})$.
The claim for $\sfd$ immediately follows from the triangle inequality.
For what concerns $m$ let us set $\alpha:=2(\pi-2\rho)>0$.
For $x\in \Omega$,
we define $\Omega_{u,\alpha,x},\Omega_{v,\alpha,x},\Omega_{u,v,\alpha,x}$ as in \eqref{eq:tecset}.
For all $y\in \Omega \setminus \Omega_{u,v,\alpha,x}$,
we have
\begin{align*}
\sfd_\Y(u(x),v(x))+\sfd_\Y(v(x),u(y))+\sfd_\Y(u(y),u(x))< 4\rho+\alpha=2\pi,\\
\sfd_\Y(v(y),u(y))+\sfd_\Y(u(y),v(x))+\sfd_\Y(v(x),v(y))< 4\rho+\alpha=2\pi.
\end{align*}
Therefore,
by using \eqref{eq:rough3} for $l=2\rho$ twice via the midpoint of $u(y)$ and $v(x)$,
for all $y\in \Omega \setminus \Omega_{u,v,\alpha,x}$,
\begin{equation*}
\sfd_\Y(m(x),m(y))\leq c_{\rho}(\sfd_\Y(u(x),u(y))+\sfd_\Y(v(x),v(y))
\end{equation*}
for some constant $c_\rho>0$ depending only on $\rho$.
This implies
\begin{align}\label{eq:energy11}
\kse^2_{2,r}[m,\Omega](x)&\leq \frac{2c^2_\rho}{\mm(B_r(x))} \int_{B_r(x) \setminus \Omega_{u,v,\alpha,x}}\, \left(\frac{\sfd^2_\Y(u(x),u(y))}{r^2}+\frac{\sfd^2_\Y(v(x),v(y))}{r^2}\right)\,\d \mm(y)\\ \notag
&+ \frac{1}{\mm(B_r(x))} \int_{B_r(x) \cap \Omega_{u,v,\alpha,x}}\, \frac{\sfd^2_\Y(m(x),m(y))}{r^2}\,\d \mm(y)\\ \notag
&\leq 2c^2_\rho\,\left( \kse^2_{2,r}[u,\Omega](x)+\kse^2_{2,r}[v,\Omega](x)\right)+ \frac{4\rho^2}{r^2} \frac{\mm(B_r(x) \cap \Omega_{u,v,\alpha,x})}{\mm(B_r(x))}\\ \notag
&\leq \left(2c^2_\rho+\frac{4\rho^2}{\alpha^2}\right)\,\left( \kse^2_{2,r}[u,\Omega](x)+\kse^2_{2,r}[v,\Omega](x)\right),
\end{align}
where we used
\begin{align*}
\frac{\alpha^2}{r^2} \frac{\mm(B_r(x)\cap \Omega_{u,v,\alpha,x})}{\mm(B_r(x))}&\leq \frac{\alpha^2}{r^2} \left(\frac{\mm(B_r(x)\cap \Omega_{u,\alpha,x})}{\mm(B_r(x))}+\frac{\mm(B_r(x)\cap \Omega_{v,\alpha,x})}{\mm(B_r(x))}\right)\\
&\leq \kse^2_{2,r}[u,\Omega](x)+\kse^2_{2,r}[v,\Omega](x).
\end{align*}
Integrating \eqref{eq:energy11} over $\Omega$,
and letting $r\downarrow 0$,
we deduce $m \in \ks^{1,2}(\Omega,\bar{\B}_{\rho}(\o))$.

We prove \eqref{eq:energy12}.
Now let us set $\beta:=\pi-2\rho>0$ and for $x\in \Omega$ define $\Omega_{u,\beta,x},\Omega_{v,\beta,x},\Omega_{u,v,\beta,x}$ as in \eqref{eq:tecset} again.
For all $y\in \Omega \setminus \Omega_{u,v,\beta,x}$,
there holds
\begin{equation*}
\sfd_\Y(u(x),v(x))+\sfd_\Y(v(x),v(y))+\sfd_\Y(v(y),u(y))+\sfd_\Y(u(y),u(x))< 2\beta+4\rho=2\pi.
\end{equation*}
Lemma \ref{lem:cos1} tells us that for all $y\in \Omega \setminus \Omega_{u,v,\beta,x}$,
\begin{align*}
\cos^2\frac{\sfd(x)}{2}&\,\sfd^2_\Y(m(x),m(y))+\frac{1}{4}(\sfd(x)-\sfd(y))^2\\
&\leq \frac{1}{2} \left(\sfd^2_\Y(u(x),u(y))+\sfd^2_\Y(v(x),v(y)) \right)\\ 
&+\Cub(\sfd_\Y(u(x),u(y)),\sfd_\Y(v(x),v(y)),\sfd(x)-\sfd(y),\sfd_\Y(m(x),m(y))).
\end{align*}
We shall write the last term as $\Cub(x,y)$ for short.
It follows that
\begin{align}\label{eq:desired}
\cos^2\frac{\sfd(x)}{2}&\,\kse^2_{2,r}[m,\Omega;u,v,\beta](x)+\frac{1}{4}\kse^2_{2,r}[\sfd,\Omega;u,v,\beta](x)\\ \notag
&\leq \frac{1}{2}\left( \kse^2_{2,r}[u,\Omega](x)+\kse^2_{2,r}[v,\Omega](x) \right)+\frac{1}{\mm(B_r(x))}\int_{B_r(x) \setminus \Omega_{u,v,\beta,x}}\,\frac{\Cub(x,y)}{r^2}\,\d \mm(y), \notag
\end{align}
where $\kse^2_{2,r}[m,\Omega;u,v,\beta](x)$ and $\kse^2_{2,r}[\sfd,\Omega;u,v,\beta](x)$ are defined as \eqref{eq:tecfct}.

By Theorem \ref{thm:openenerygy} and Lemma \ref{lem:tech},
and by the same argument in the proof of Lemma \ref{lem:tech},
$\mm$-a.e. $x\in \Omega$ satisfies the following properties:
\begin{enumerate}\setlength{\itemsep}{+0.7mm}
\item $\kse_{2,r}[m,\Omega;u,v,\beta](x)\to \e_2[m](x), \kse_{2,r}[\sfd,\Omega;u,v,\beta](x) \to \e_2[\sfd](x)$ as $r\downarrow0$;
\item $\kse_{2,r}[u,\Omega](x)\to \e_2[u](x), \kse_{2,r}[v,\Omega](x)\to \e_2[v](x)$ as $r\downarrow0$;
\item $\lip u(x),\lip v(x),\lip \sfd(x),\lip m(x)<+\infty$.
\end{enumerate}
We now fix such a point $x\in \Omega$.
Let us verify
\begin{equation}\label{eq:error term}
\frac{1}{\mm(B_r(x))}\int_{B_r(x) \setminus \Omega_{u,v,\beta,x}}\,\frac{\Cub(x,y)}{r^2}\,\d \mm(y)\to 0
\end{equation}
as $r\downarrow0$.
Each term in \eqref{eq:error term} can be written as
\begin{equation}\label{eq:error term2}
\frac{f(x)}{\mm(B_r(x))}\int_{B_r(x) \setminus \Omega_{u,v,\beta,x}}\,\frac{\sfd^{\theta_1}_\Y(u(x),u(y))\sfd^{\theta_2}_\Y(v(x),v(y))(\sfd(x)-\sfd(y))^{\theta_3}\sfd^{\theta_4}_\Y(m(x),m(y))}{r^2}\,\d \mm(y)
\end{equation}
for $\theta_1,\theta_2,\theta_3,\theta_4\geq 0$ such that $\theta_1+\theta_2+\theta_3+\theta_4\geq 3$ and at least one of $i=1,\dots,4$ has $\theta_i\geq 2$ (see Remark \ref{rem:error}).
Let us assume $\theta_3\geq 2$ since the other cases can be handled in a similar way.
For the absolute value of \eqref{eq:error term2} we have
\begin{align}\label{eq:error3}
&\frac{|f(x)|}{\mm(B_r(x))}\int_{B_r(x)\setminus \Omega_{u,v,\beta,x}}\,\frac{\sfd^{\theta_1}_\Y(u(x),u(y))\sfd^{\theta_2}_\Y(v(x),v(y)) |\sfd(x)-\sfd(y)|^{\theta_3}\sfd^{\theta_4}_\Y(m(x),m(y))}{r^2}\,\d \mm(y)\\ \notag
&\leq |f(x)|\,r^{\theta_1+\theta_2+\theta_3+\theta_4-2}\,\kse^2_{2,r}[\sfd,\Omega;u,v,\beta](x) \,I,
\end{align}
where
\begin{equation*}
I:=\sup_{y\in B_r(x)\setminus \{x\}} \left(\frac{\sfd^{\theta_1}_\Y(u(x),u(y))}{\sfd^{\theta_1}(x,y)}\frac{\sfd^{\theta_2}_\Y(u(x),u(y))}{\sfd^{\theta_2}(x,y)}\frac{|\sfd(x)-\sfd(y)|^{\theta_{3}-2}}{\sfd^{\theta_3-2}(x,y)}\frac{\sfd^{\theta_4}_\Y(m(x),m(y))}{\sfd^{\theta_4}(x,y)}\right).
\end{equation*}
By $\theta_1+\theta_2+\theta_3+\theta_4-2\geq 1$ and by the choice of $x$,
the right hand side of \eqref{eq:error3} goes to $0$ as $r\downarrow 0$.
Thus,
we confirm the validity of \eqref{eq:error term}.
Letting $r\downarrow 0$ in \eqref{eq:desired},
we can derive the desired assertion.
\end{proof}

We next prove the following (cf. \cite[Lemma 1.14]{S}):
\begin{lem}\label{lem:energy2}
Assume that
$(\X,\sfd,\mm)$ is infinitesimally Hilbertian.
For $u\in \ks^{1,2}(\Omega,\bar{\B}_{\rho}(\o)),\eta\in \ks^{1,2}(\Omega,[0,1])$,
we define $u_{\eta}:\Omega \to \bar{\B}_{\rho}(\o)$ by $u_{\eta}:=\G^{u,\o}_{\eta}$ and $\sfd:=\sfd_\Y(u,\o)$.
Then we have $u_{\eta}\in \ks^{1,2}(\Omega,\bar{\B}_{\rho}(\o)), \sfd \in \ks^{1,2}(\Omega,\mathbb{R})$ and 
\begin{equation*}
\e^2_{2}[u_{\eta}]\leq \frac{\sin^2(1-\eta)\sfd}{\sin^2 \sfd }  \left( \e^2_{2}[u]-\e^2_{2}[\sfd] \right)+\e^2_2[(1-\eta)\sfd]\quad \text{$\mm$-a.e. on $\Omega$}.
\end{equation*}
\end{lem}
\begin{proof}
Let us verify $u_{\eta}\in \ks^{1,2}(\Omega,\bar{\B}_{\rho}(\o)), \sfd \in \ks^{1,2}(\Omega,\mathbb{R})$.
Similarly to the proof of Lemma \ref{lem:energy1},
the thesis for $\sfd$ follows from the triangle inequality.
By virtue of \eqref{eq:rough3} for $l=\rho$,
for all $x,y\in \Omega$,
\begin{align*}
\sfd_\Y(u_\eta(x),u_{\eta}(y))&\leq c_{\rho}\,\sfd_\Y(u(x),u(y))+|\eta(x)-\eta(y)|\sfd_\Y(u(y),\o)\\
&\leq c_{\rho}\,\sfd_\Y(u(x),u(y))+\rho|\eta(x)-\eta(y)|
\end{align*}
for some constant $c_\rho>0$ depending only on $\rho$.
Therefore,
\begin{equation*}
\kse^2_{2,r}[u_{\eta},\Omega](x)\leq 2c^2_{\rho}\,\kse^2_{2,r}[u,\Omega](x)+2\rho^2\kse^2_{2,r}[\eta,\Omega](x).
\end{equation*}
Integrating this estimate over $\Omega$,
and letting $r\downarrow 0$,
we conclude $u_\eta \in \ks^{1,2}(\Omega,\bar{\B}_{\rho}(\o))$.

Using Lemma \ref{lem:cos2},
for all $x,y\in \Omega$ we have
\begin{align*}
\sfd^2_\Y(u_{\eta}(x),u_{\eta}(y))&\leq \frac{\sin^2(1-\eta(x))\sfd(x)}{\sin^2 \sfd(x) }  \left(\sfd^2_{\Y}(u(x),u(y))- (\sfd(y)-\sfd(x))^2 \right)\\
&+(1-\eta(x))^2(\sfd(y)-\sfd(x))^2+\sfd^2(x)(\eta(y)-\eta(x))^2\\
&-2(1-\eta(x))\sfd(x)(\eta(y)-\eta(x))(\sfd(y)-\sfd(x))\\
&+\Cub(\eta(y)-\eta(x),\sfd_\Y(u(x),u(y)),\sfd(y)-\sfd(x),\sfd_\Y(u_{\eta}(x),u_{\eta}(y))).
\end{align*}
Setting $\xi:=1-\eta$,
we also have
\begin{align}\label{eq:energy51}
\sfd^2_\Y(u_{\eta}(x),u_{\eta}(y))&\leq \frac{\sin^2\xi(x)\sfd(x)}{\sin^2 \sfd(x) }  \left(\sfd^2_{\Y}(u(x),u(y))-(\sfd(y)-\sfd(x))^2 \right)\\ \notag
&+\xi^2(x)(\sfd(y)-\sfd(x))^2+\sfd^2(x)(\xi(y)-\xi(x))^2\\ \notag
&+2\xi(x)\sfd(x)(\xi(y)-\xi(x))(\sfd(y)-\sfd(x))\\ \notag
&+\Cub(\xi(x)-\xi(y),\sfd_\Y(u(x),u(y)),\sfd(y)-\sfd(x),\sfd_\Y(u_{\eta}(x),u_{\eta}(y)))\\ \notag
&= \frac{\sin^2\xi(x)\sfd(x)}{\sin^2 \sfd(x) }  \left(\sfd^2_{Y}(u(x),u(y))-(\sfd(y)-\sfd(x))^2 \right)\\ \notag
&+\xi^2(x)(\sfd(y)-\sfd(x))^2+\sfd^2(x)(\xi(y)-\xi(x))^2\\ \notag
&+\frac{\xi(x)\sfd(x)}{2}\left[ \{(\xi+\sfd)(x)-(\xi+\sfd)(y)\}^2- \{(\xi-\sfd)(x)-(\xi-\sfd)(y)\}^2  \right]\\ \notag
&+\Cub(\xi(x)-\xi(y),\sfd_\Y(u(x),u(y)),\sfd(y)-\sfd(x),\sfd_\Y(u_{\eta}(x),u_{\eta}(y))).
\end{align}
Dividing the later expression by $r^2$
and integrating in $y$ over $B_r(x)$ lead to
\begin{align*}
\kse^2_{2,r}[u_{\eta},\Omega](x)&\leq \frac{\sin^2\xi(x)\sfd(x)}{\sin^2 \sfd(x) }  \left( \kse^2_{2,r}[u,\Omega](x)-\kse^2_{2,r}[\sfd,\Omega](x) \right)\\
&+\xi^2(x)\kse^2_{2,r}[\sfd,\Omega](x)+\sfd^2(x)\kse^2_{2,r}[\xi,\Omega](x)\\
&+\frac{\xi(x)\sfd(x)}{2}\left( \kse^2_{2,r}[\xi+\sfd,\Omega](x)-\kse^2_{2,r}[\xi-\sfd,\Omega](x)  \right)+\fint_{B_r(x)}\,\frac{\Cub(x,y)}{r^2}\,\d \mm(y),
\end{align*}
where $\Cub(x,y)$ denotes the cubic terms in \eqref{eq:energy51}.
In the same manner as in the proof of Lemma \ref{lem:energy1},
we see
\begin{equation*}
\fint_{B_r(x)}\,\frac{\Cub(x,y)}{r^2}\,\d \mm(y) \to 0 \quad \text{$\mm$-a.e. $x\in \Omega$}
\end{equation*}
as $r\downarrow 0$.
By virtue of Theorem \ref{thm:openenerygy},
we obtain
\begin{align*}
\e^2_{2}[u_{\eta}]&\leq \frac{\sin^2\xi \sfd}{\sin^2 \sfd }  \left( \e^2_{2}[u]-\e^2_{2}[\sfd] \right)+\xi^2 \e^2_{2}[\sfd]+\sfd^2\e^2_{2}[\xi]+\frac{\xi \sfd}{2}\left( \e^2_{2}[\xi+\sfd]-\e^2_{2}[\xi-\sfd]  \right)\quad \text{$\mm$-a.e. on $\Omega$}.
\end{align*}

Now,
it suffices to prove that
\begin{equation*}
\xi^2 \e^2_{2}[\sfd]+\sfd^2\e^2_{2}[\xi]+\frac{\xi \sfd}{2}\left( \e^2_{2}[\xi+\sfd]-\e^2_{2}[\xi-\sfd]  \right)=\e^2_2[\xi \sfd] \quad \text{$\mm$-a.e. on $\Omega$}.
\end{equation*}
Thanks to Proposition \ref{thm:consR},
for any $u\in \ks^{1,2}(\Omega,\mathbb{R})$,
we have $|D u|=c_d\, \e_2[u]$ $\mm$-a.e. on $\Omega$,
and hence it is enough to show that
\begin{equation}\label{eq:energy2suf}
\xi^2 |D \sfd|^2+\sfd^2|D \xi|^2+\frac{\xi \sfd}{2}\left( |D(\xi+\sfd)|^2-|D(\xi-\sfd)|^2  \right)=|D(\xi \sfd)|^2 \quad \text{$\mm$-a.e. on $\Omega$}.
\end{equation}
Since $(\X,\sfd,\mm)$ is infinitesimally Hilbertian,
Theorem \ref{thm:IH} tells us that
\begin{equation}\label{eq:energy2lei}
|D(\xi \sfd)|^2=\langle \nabla (\xi \sfd),\nabla(\xi \sfd) \rangle=\xi^2 |D \sfd|^2+\sfd^2|\nabla \xi|^2+2\xi \sfd\langle \nabla \xi,\nabla \sfd \rangle\quad \text{$\mm$-a.e. on $\Omega$},
\end{equation}
where we used the Leibniz rule.
Furthermore,
by the linearity of the gradient,
\begin{equation}\label{eq:energy2lin}
|D(\xi+\sfd)|^2-|D(\xi-\sfd)|^2=\langle \nabla (\xi +\sfd),\nabla(\xi +\sfd) \rangle-\langle \nabla (\xi -\sfd),\nabla(\xi -\sfd) \rangle=4\langle \nabla \xi,\nabla \sfd \rangle \quad \text{$\mm$-a.e. on $\Omega$}.
\end{equation}
Combining \eqref{eq:energy2lei} and \eqref{eq:energy2lin} yields
\begin{equation*}
|D(\xi \sfd)|^2=\xi^2 |D \sfd|^2+\sfd^2|D \xi|^2+\frac{\xi \sfd}{2}\left( |D(\xi+\sfd)|^2-|D(\xi-\sfd)|^2  \right)\quad \text{$\mm$-a.e. on $\Omega$},
\end{equation*}
and this is nothing but \eqref{eq:energy2suf},
thus concluding the proof.
\end{proof}

We now prove that the energy $\E^{\Omega}_{2}$ satisfies a certain kind of convexity (cf. \cite[Proposition 1.15]{GT2}):
\begin{prop}\label{prop:energy3}
Assume that
$(\X,\sfd,\mm)$ is infinitesimally Hilbertian.
For $u,v\in \ks^{1,2}(\Omega,\bar{\B}_{\rho}(\o))$,
we set $m:=\G^{u,v}_{1/2}\in \ks^{1,2}(\Omega,\bar{\B}_{\rho}(\o))$ and $\sfd:=\sfd_\Y(u,v),\dfr:=\sfd_{\Y}(m,\o)\in \ks^{1,2}(\Omega,\mathbb{R})$.
Let $\eta \in \ks^{1,2}(\Omega,[0,1])$ be a function determined by solving
\begin{equation}\label{eq:intermediate}
\frac{\sin(1-\eta)\dfr}{\sin \dfr}=\cos\frac{\sfd}{2},
\end{equation}
and define $m_\eta:=\G^{m,\o}_{\eta}\in \ks^{1,2}(\Omega,\bar{\B}_{\rho}(\o))$.
Then we have
\begin{equation*}
\E^{\Omega}_{2}(m_{\eta})+\cos^8 \rho \,\E^{\Omega}_{2}\left(\frac{\tan \frac{\sfd}{2}}{\cos \dfr}\right)\leq \frac{1}{2}\E^{\Omega}_{2}(u)+\frac{1}{2}\E^{\Omega}_{2}(v).
\end{equation*}
\end{prop}
\begin{proof}
Combining Lemmas \ref{lem:energy1} and \ref{lem:energy2},
we have
\begin{align}\label{eq:energy31}
\e^2_2[m_{\eta}]&\leq \frac{\sin^2 (1-\eta)\dfr}{\sin^2 \dfr}(\e^2_2[m]-\e^2_2[\dfr])+\e^2_2[(1-\eta)\dfr]\\ \notag
&=\cos^2 \frac{\sfd}{2}(\e^2_2[m]-\e^2_2[\dfr])+\e^2_2[(1-\eta)\dfr]\\ \notag
&\leq \frac{1}{2}\left( \e^2_2[u]+\e^2_{2}[v] \right)-\frac{1}{4}\e^2_2[\sfd]-\cos^2 \frac{\sfd}{2}\,\e^2_2[\dfr]+\e^2_2[(1-\eta)\dfr] \quad \text{$\mm$-a.e. on $\Omega$.}
\end{align}

We now observe that the following holds:
\begin{equation}\label{eq:energy32}
\e^2_2[(1-\eta)\dfr]+\frac{\cos^4 \frac{\sfd}{2}\,\cos^4 \dfr}{1-\cos^2 \frac{\sfd}{2}\,\sin^2 \dfr}\,\e^2_2 \left[\frac{\tan \frac{\sfd}{2}}{ \cos \dfr} \right]=\cos^2 \frac{\sfd}{2}\,\e^2_2[\dfr]+\frac{1}{4}\e^2_2[\sfd] \quad \text{$\mm$-a.e. on $\Omega$.}
\end{equation}
In view of Proposition \ref{thm:consR},
this is equivalent to the following:
\begin{equation*}
|D((1-\eta)\dfr)|^2+\frac{\cos^4 \frac{\sfd}{2}\,\cos^4 \dfr}{1-\cos^2 \frac{\sfd}{2}\,\sin^2 \dfr}\,\left|D\left(\frac{\tan \frac{\sfd}{2}}{ \cos \dfr}\right)\right|^2=\cos^2 \frac{\sfd}{2}\,|D\dfr|^2+\frac{1}{4}|D \sfd|^2 \quad \text{$\mm$-a.e. on $\Omega$.}
\end{equation*}
Since $(\X,\sfd,\mm)$ is infinitesimally Hilbertian,
it can be derived from Theorems \ref{thm:gradchain}, \ref{thm:IH} (chain rule, Leibniz rule), the definition of $\eta$, and a straightforward calculation.
Combining \eqref{eq:energy31} and \eqref{eq:energy32},
we obtain
\begin{align*}
\e^2_2[m_{\eta}]+\cos^8 \rho\,\e^2_2 \left[\frac{\tan \frac{\sfd}{2}}{ \cos \dfr} \right]
&\leq \e^2_2[m_{\eta}]+\cos^4 \frac{\sfd}{2}\,\cos^4 \dfr\,\e^2_2 \left[\frac{\tan \frac{\sfd}{2}}{ \cos \dfr} \right]\\
&\leq \e^2_2[m_{\eta}]+\frac{\cos^4 \frac{\sfd}{2}\,\cos^4 \dfr}{1-\cos^2 \frac{\sfd}{2}\,\sin^2 \dfr}\,\e^2_2 \left[\frac{\tan \frac{\sfd}{2}}{ \cos \dfr} \right]\\
&\leq \frac{1}{2}\left( \e^2_2[u]+\e^2_{2}[v] \right) \quad \text{$\mm$-a.e. on $\Omega$.}
\end{align*}
By integrating this inequality over $\Omega$,
we obtain the desired one.
\end{proof}

\subsection{Proof of Theorem \ref{thm:main}}\label{sec:pf}

We are now in a position to prove Theorem \ref{thm:main}.
\begin{proof}[Proof of Theorem \ref{thm:main}]
Let $(\X,\sfd,\mm)$ be an infinitesimally Hilbertian, locally uniformly doubling, strongly rectifiable space satisfying a Poincar\'e inequality,
and let $\Omega$ be a bounded open subset of $\X$ with $\mm(\X\setminus \Omega)>0$.
Let $\Y_{\o}=(\Y,\sfd_\Y,\o)$ be a pointed $\CAT(1)$ space,
and let $\bar{\B}_{\rho}(\o)$ be a regular ball.
Let $(u_n)_n\subset \ks^{1,2}_{\bar u}(\Omega,\bar{\B}_{\rho}(\o))$ stand for a minimizing sequence of $\E^{\Omega}_{2,\bar{u}}$.
By Proposition \ref{prop:Dirichlet},
the functional $\E_{2,\bar u}^\Omega$ is lower semicontinuous,
and hence
it is sufficient to show that $(u_n)_n$ is an $L^2(\Omega,\bar{\B}_{\rho}(\o))$-Cauchy sequence.

Set $I:=\lim_n \E_{2,\bar u}^\Omega(u_n)=\inf  \E_{2,\bar u}^\Omega$.
We also set
\begin{equation*}
m_{n,m}:=\G^{u_n,u_m}_{1/2},\quad \sfd_{n,m}:=\sfd_\Y(u_n,u_m),\quad \dfr_{n,m}:=\sfd_\Y(m_{n,m},\o),\quad m_{n,m,\eta}:=\G^{m_{n,m},\o}_{\eta_{n,m}},
\end{equation*}
where $\eta_{n,m}\in \ks^{1,2}(\Omega,[0,1])$ is defined as \eqref{eq:intermediate}.
Proposition \ref{prop:energy3} yields
\begin{align*}
\cos^8 \rho \,\E^{\Omega}_{2,\bar{u}}\left(\frac{\tan \frac{\sfd_{n,m}}{2}}{\cos \dfr_{n,m}}\right)&\leq \frac{1}{2}\E^{\Omega}_{2,\bar{u}}(u_{n})+\frac{1}{2}\E^{\Omega}_{2,\bar{u}}(u_m)-\E^{\Omega}_{2,\bar{u}}(m_{n,m,\eta})\\
&\leq \frac{1}{2}\E^{\Omega}_{2,\bar{u}}(u_{n})+\frac{1}{2}\E^{\Omega}_{2,\bar{u}}(u_m)-I;
\end{align*}
in particular,
\begin{equation*}
\lims_{n,m\to\infty}\int_\Omega \left|D\left(\frac{\tan \frac{\sfd_{n,m}}{2}}{\cos \dfr_{n,m}}\right) \right|^2\,\d\mm=0.
\end{equation*}
Furthermore,
a Poincar\'e inequality under Dirichlet boundary conditions (see e.g., \cite[Lemma 6.3]{GT2}, \cite[Subsection 5.5]{BB}) together with Proposition \ref{prop:Dirichlet} leads us to
\begin{equation*}
\lim_{n,m\to\infty}\int_\Omega \left|\frac{\tan \frac{\sfd_{n,m}}{2}}{\cos \dfr_{n,m}} \right|^2\,\d\mm=0.
\end{equation*}
Thus,
$\lim_{n,m\to\infty}\int_\Omega |\sfd_{n,m}|^2\,\d\mm=0$ since $|\sfd_{n,m}|^2 \leq 4|\tan(\sfd_{n,m}/2)|^2$ and $|\cos \dfr_{n,m}|^2\leq 1$ over $\Omega$.
This concludes the proof.
\end{proof}

\subsection*{{\rm Acknowledgements}} 
The author is grateful to Keita Kunikawa for fruitful discussions during this work.
The author expresses his gratitude to Shouhei Honda for valuable comments.
The author would like to thank the anonymous referee for useful comments.
The author was supported by JSPS KAKENHI (JP23K12967).


\end{document}